\documentclass{stijl}     

\usepackage{amssymb} 
\usepackage{amsmath} 
\usepackage{stmaryrd}
\usepackage{euscript} 
\usepackage{latexsym} 
\usepackage{mathrsfs} 
\usepackage{wasysym}
\usepackage{float} 
\usepackage{xcolor} 
\usepackage{enumitem}
\usepackage{proof}
\usepackage{url}

\newcommand{\IPC}{{\sf IPC}}

\newcommand{\K}{{\sf K}}

\newcommand{\KD}{{\sf KD}}

\newcommand{\GL}{{\sf GL}}
\newcommand{\SL}{{\sf SL}}

\newcommand{\LJ}{{\sf G3ip}}

\newcommand{\LJX}{{\sf G3iX}}

\newcommand{\LJK}{{\sf G3iK}_\Box}

\newcommand{\DY}{{\sf G4ip}}

\newcommand{\DYX}{{\sf G4iX}}

\newcommand{\DYKf}{{\sf G4iK4}_\Box}
\newcommand{\DYSf}{{\sf G4iS4}_\Box}

\newcommand{\G}{{\sf G}}

\newcommand{\rsch}{{\EuScript R}}

\newcommand{\lgc}{{\sf L}}
\newcommand{\lang}{\ensuremath {{\EuScript L}}}
\newcommand{\defn}{\equiv _{\mbox{\em \tiny df}}} 
\newcommand{\af}{\vdash}

\newcommand{\imp}{\rightarrow}

\newcommand{\en}{\wedge} 
\newcommand{\of}{\vee}

\newcommand{\mdl}{\raisebox{0mm}{$\ocircle$}} 
\newcommand{\bx}{\raisebox{.1mm}{$\Box$}}
\newcommand{\bof}{\bigvee}
\newcommand{\ben}{\bigwedge}
\newcommand{\seq}{\Rightarrow}
\newcommand{\sml}{\ll}

\newcommand{\MR}{{\cal R}}

\newcommand{\De}{\Delta}
\newcommand{\Ga}{\Gamma}

\newcommand{\Sig}{\Sigma}

\newcommand{\gam}{\gamma}
\renewcommand{\phi}{\varphi}
\newcommand{\varchi}{\raisebox{.4ex}{\mbox{$\chi$}}}

\newcommand{\cald}{{\EuScript D}}

\newtheorem{Theor}{Theorem}[section]
\newenvironment{theorem}{\begin{Theor}\em }{\end{Theor}}
\newtheorem{Lemma}[Theor]{Lemma}
\newenvironment{lemma}{\begin{Lemma}\em }{\end{Lemma}}
\newtheorem{Coro}[Theor]{Corollary}
\newenvironment{corollary}{\begin{Coro}\em }{\end{Coro}}
\newtheorem{Remark}[Theor]{Remark}
\newenvironment{remark}{\begin{Remark}\em }{\end{Remark}}
\newtheorem{defin}[Theor]{Definition}

\newtheorem{exam}[Theor]{Example}
\newenvironment{example}{\begin{exam}\em }{\end{exam}}
\newenvironment{proof}{{\bf Proof}}{\hfill $\slot$}
\newcommand{\slot}{\hfill \mbox{$\dashv$}}

\begin{document}

\title{The G4i analogue of a G3i sequent calculus}
\vskip5pt 
\author{
Rosalie Iemhoff
\footnote{Utrecht University, the Netherlands, r.iemhoff@uu.nl. Support by the Netherlands Organisation for Scientific Research under grant 639.073.807 as well as partial support by the MOSAIC project (EU H2020-MSCA-RISE-2020 Project 101007627) is gratefully acknowledged.} }
\date{}
   
\maketitle

\begin{abstract}
This paper provides a method to obtain terminating analytic calculi for a large class of intuitionistic modal logics. For a given logic \lgc\ with a cut-free calculus \G\ that is an extension of \LJ\ the method produces a terminating analytic calculus that is an extension of \DY\ and equivalent to \G. \DY\ was  introduced by Roy Dyckhoff in 1992 as a terminating analogue of the calculus \LJ\ for intuitionistic propositional logic. Thus this paper can be viewed as an extension of Dyckhoff's work to intuitionistic modal logic. 

\end{abstract}

\section{Introduction}
One of the standard calculi without structural rules for $\IPC$ is $\LJ$ (Figure~\ref{figgthm}), which is the propositional part of the calculus {\sf G3i} from \cite{troelstra&schwichtenberg96}. This is an elegant analytic calculus, but it has the unfortunate feature that unrestricted proof search is not terminating in it. The reason for this lies in its left implication rule, in which the principal formula occurs in one of the premises. In \cite{dyckhoff92} Roy Dyckhoff introduced a calculus \DY\footnote{Originally, the calculus was called {\sf LJT} by Dyckhoff, but in this paper we use the name $\DY$ from \cite{troelstra&schwichtenberg96}, which seems the more common terminology nowadays.} (Figure~\ref{figdy}) that is the result of replacing the single left implication rule in $\LJ$ by four left implication rules, each corresponding to the outermost logical symbol of the antecedent of the principal implication. This system was shown to be terminating and equivalent to $\LJ$, where a calculus is terminating if there exists an order on sequents under which in all rules the premises come before the conclusion in that order. 

The modest aim in this short note is to extend Dyckhoff's result to intuitionistic modal logics. For any extension $\LJX$ of \LJ\ by modal rules $\MR$, a calculus, called  \DYX, is defined that is an extension of \DY\ by $\MR$ and several additional rules determined by $\MR$. It is shown that under some mild closure conditions $\LJX$ and $\DYX$ are equivalent. Moreover, if the rules in $\MR$ are terminating in an order that is an extension (to modal logic) of Dyckhoff's original order on sequents, then $\DYX$ terminates in that order too. Thus for a logic that has a calculus \LJX\ that is an extension of \LJ\ by terminating rules, establishing the equivalence of \LJX\ and \DYX\ indeed is a method to obtain a terminating sequent calculus for that logic.

The interest in terminating calculi lies in the fact that they can be a useful tool in establishing certain properties of a logic, such as decidability, or uniform interpolation, where the syntactic approach developed in \cite{iemhoff17,jalali&tabatabai2018} defines interpolants on the basis of the rules of a calculus, for which it is essential that the calculus is terminating. In fact, its potential use in syntactic proofs of uniform interpolation was the main motivation for developing the method in this paper. This method to obtain a terminating calculus for a given logic uses a (usually nonterminating) calculus based on \LJ. Since many intuitionistic modal logics have a calculus based on \LJ, we therefore hope that this paper will be a convenient tool in settings where terminating calculi are required. 

The general result presented in this paper grew out of the research carried out in \cite{giessen&iemhoff2019,giessen&iemhoff2020,iemhoff17}, in which \DY\ based calculi are developed for four concrete logics, namely the intuitionistic versions of the logics \K, \KD, G\"odel-L\"ob Logic \GL\ and Strong L\"ob Logic \SL. The results for the first two logics follow from the results in this note, while the results for the other two logics require a slight adjustment of the method presented in this note, as explained in Section~\ref{secconclusion}.   

The paper is built-up as follows. Section~\ref{seclogics} introduces \LJX\ and \DYX, the order $\sml$ on sequents based on Dyckhoff's weight function, and the further notions needed in the paper. Section~\ref{secequivalence} contains the main theorem, the equivalence of \LJX\ and \DYX. Section~\ref{secconclusion} summarizes the results and discusses possible extensions.

\section{Preliminaries}
 \label{seclogics}
We consider (modal) propositional logics in a language $\lang$ that contains a {\em constant} $\bot$, {\em propositional variables} or {\em atoms} $p,q,r,\dots$, {\em modal operators} $\mdl_1,\mdl_2,\dots$ and the {\em connectives} $\neg,\en,\of,\imp$, where $\neg\phi$ is defined as $(\phi \imp\bot)$. $\bot$ is by definition not an atom. $\mdl$ ranges over the modal operators and a formula of the form $\mdl \phi$ is a {\em pure modal formula}. We use the convention that $\en$ and $\of$ bind stronger than $\imp$. 

We denote finite multisets of formulas by $\Ga,\Pi,\De,\Sig$. 
We denote by $\Ga \cup \Pi$ the multiset that contains only formulas $\phi$ that belong to $\Ga$ or $\Pi$ and the number of occurrences of $\phi$ in $\Ga \cup \Pi$ is the sum of the occurrences of $\phi$ in $\Ga$ and in $\Pi$. Furthermore ($a$ for antecedent, $s$ for succedent):  
\[
 (\Ga \seq \De)^a \defn \Ga \ \ \ \ (\Ga \seq \De)^s \defn \De \ \ \ \ 
 \mdl\Ga \defn \{ \mdl\phi \mid \phi \in \Ga \}.
\]
We only consider single-conclusion sequents, which are expressions $(\Ga \seq \De)$, where $\De$ contains at most one formula, that are interpreted as $I(\Ga \seq \De) = (\ben\Ga \imp \bof\De)$.  In a sequent, $\Ga,\Pi$ is short for $\Ga\cup \Pi$.   
When sequents are used in the setting of formulas, we often write $S$ for $I(S)$, such as in $\af \bof_i S_i$, which thus denotes $\af \bof_i I(S_i)$. 

The {\em degree} of a formula $\phi$ is inductively defined by $d(\bot)=0$, $d(p)=1$, $d(\mdl_i \phi)=d(\phi)+1$, and $d(\phi\circ \psi)= d(\phi)+d(\psi)+1$ for $\circ \in\{\en,\of,\imp\}$.
In the setting of $\DYX$ systems we need an order on sequents based on a {\em weight function}, which is a function $w(\cdot)$ that assigns positive numbers to formulas in such a way that all atoms and $\bot$ have weight 1 and all other formulas have a weight above 1. 
With a weight function $w$ we associate the following order on sequents: 
$S_0 \sml_w S_1$ if and only if 
$S_0^a\cup S_0^s \sml_w S_1^a\cup S_1^s$, where $\sml_w$ is the order on multisets determined by $w$ as in \cite{dershowitz&manna79} (where they in fact define $\gg$): for multisets $\Ga,\De$ we have $\De \sml_w \Ga$ if $\De$ is the result of replacing one or more formulas in $\Ga$ by zero or more formulas of lower weight. 

\subsection{Calculi G3iX and G4iX}
 \label{secclassical}
 
\begin{figure}
 \centering
\[\small 
 \begin{array}{ll}
 \deduce{\Ga,p \seq p}{} \ \ \ \text{{\it Ax} \ \ ($p$ an atom)}  & 
  \deduce{\Ga,\bot\seq \De}{} \ \ \ L\bot \\
 \\ 
 \infer[R\en]{\Ga \seq \phi \en \psi}{\Ga\seq \phi & \Ga \seq \psi} & 
  \infer[L\en]{\Ga, \phi\en \psi \seq \De}{\Ga, \phi, \psi \seq \De} \\
 \\
 \infer[R\!\of \ (i=0,1)]{\Ga \seq \phi_0 \of \phi_1}{\Ga \seq \phi_i} & 
  \infer[L\of]{\Ga,\phi\of \psi\seq \De}{\Ga, \phi \seq \De & \Ga,\psi \seq \De} \\
 \\
 \infer[R\!\imp]{\Ga \seq \phi \imp \psi}{\Ga,\phi \seq \psi} & 
 \infer[L\!\imp]{\Ga,\phi\imp\psi\seq \De}{\Ga,\phi\imp\psi\seq \phi & \Ga,\psi\seq \De}\\
 \end{array}
\] 
\caption{The Gentzen calculus $\LJ$}
 \label{figgthm}
\end{figure}

\begin{figure}
 \centering
\[\small 
 \begin{array}{lll}
 \deduce{\Ga,p \seq p}{} \ \ \ \text{{\it Ax} \ \ ($p$ an atom)}  & 
  \deduce{\Ga,\bot\seq \De}{} \ \ \ L\bot \\
 \\ 
 \infer[R\en]{\Ga \seq \phi \en \psi}{\Ga\seq \phi & \Ga \seq \psi} & 
  \infer[L\en]{\Ga, \phi\en \psi \seq \De}{\Ga, \phi, \psi \seq \De} \\
 \\
 \infer[R\!\of \ (i=0,1)]{\Ga \seq \phi_0 \of \phi_1}{\Ga \seq \phi_i} & 
  \infer[L\of]{\Ga,\phi\of \psi\seq \De}{\Ga, \phi \seq \De & \Ga,\psi \seq \De} \\
  \\
 \infer[R\!\imp]{\Ga \seq \phi \imp \psi}{\Ga,\phi \seq \psi} & 
 \infer[Lp\!\imp\text{($p$ an atom)}]{\Ga, p,p \imp \phi \seq \De}{\Ga,p,\phi \seq \De}\\
 \\
 \infer[L\en\!\imp]{\Ga, \phi\en\psi \imp \gamma \seq \De}{\Ga,\phi\imp (\psi\imp\gamma)\seq\De} & 
 \infer[L\of\!\imp]{\Ga,\phi \of \psi \imp \gamma \seq \De}{
  \Ga,\phi \imp \gamma, \psi \imp \gamma \seq \De}\\ 
 \\ 
 \infer[L\!\imp\!\imp]{\Ga, (\phi\imp \psi) \imp \gamma \seq \De}{
  \Ga, \psi\imp \gamma \seq \phi \imp \psi & \gamma,\Ga \seq \De}\\
\end{array}
\] 
\caption{The Gentzen calculus $\DY$}
 \label{figdy} 
\end{figure}

In this paper, a {\em right modal rule} is a rule of the following form, where $\mdl$ is one of the modal operators:
\[
 \infer[\rsch]{\Ga \seq \mdl\phi}{S_1 & \dots & S_n}
\]
With such a rule we associate the following implication rule
\[
  \infer[\rsch^\imp]{\Ga,\mdl\phi\imp \psi \seq \De}{
  S_1 & \dots & S_n & \Ga,\psi \seq \De}
\]
If $\LJX$ ({\sf X} a finite string) is the name of a calculus that consists of $\LJ$ plus a set of modal rules $\MR$, then $\DYX$ is the calculus $\DY$ extended by the rules in $\MR$ {\em plus} the rules $\rsch^\imp$ for those $\rsch\in \MR$ that are right modal rules. 

\begin{example}
Let $\LJX$ denote \LJ\ plus the rule 
\[
 \infer[\rsch_{\it X}]{\Pi,\mdl\Ga \seq \mdl\phi}{\mdl\Ga \seq \phi}
\]
Then $\DYX$ consists of \DY, $\rsch_{\it X}$, and the rule 
\[
 \infer[\rsch_{\it X}^\imp]{\Pi,\mdl\Ga,\mdl\phi \imp \psi \seq \De}{
  \mdl\Ga \seq \phi & \Pi,\mdl\Ga,\psi \seq \De}
\]
\end{example}

A rule is {\em nonflat} if its conclusion contains at least one connective or modal operator, and it is not an axiom, i.e\ it has nonempty premises. A set of rules or a calculus is nonflat if all of its rules that are not axioms are nonflat. Note that a calculus $\LJX$ or $\DYX$ is nonflat if all rules in $\MR$ are nonflat. 

Given a sequent calculus \G\ and a sequent $S$, $\af_\G S$ denotes that $S$ is derivable in \G.

\subsection{Terminating calculi}
Given an order $\sml$ on sequents, a rule is {\em terminating in $\sml$} if its premises come before the conclusion in order $\sml$.
A calculus is {\em terminating} if there exists a weight function $w$ such that all rules of \G\ terminate in the order $\sml_w$ on sequents (see the beginning of Section~\ref{seclogics} for the definition of $\sml_w$). 

\begin{example}
 \label{exweight}
A natural weight function based on the weight function from \cite{dyckhoff92} is inductively defined as follows: the weight of an atom and the constant $\bot$ is 1, $w_D(\phi \circ \psi) = w_D(\phi)+w_D(\psi)+i$, where $i=1$ in case $\circ \in \{\of,\imp\}$ and $i=2$ otherwise, and $w_D(\mdl_i \phi)=w_D(\phi)+1$. Let $\sml_D$ denote $\sml_{w_D}$ and call it the {\em Dyckhoff order}. It is not hard to see that all rules of \DY\ terminate in $\sml_D$. The following are examples of well-known modal rules that terminate in $\sml_D$ (the standard notation for these modalities is $\bx$, here we use $\mdl$ to be consistent with the notation in the rest of the paper). 
\[
 \infer[\rsch_{K}]{\Pi,\mdl\Ga \seq\mdl\phi}{\Ga \seq \phi} \ \ \ \ 
 \infer[\rsch_{D}]{\Pi,\mdl\Ga,\mdl\phi \seq\De}{\Ga, \phi \seq \ } \ \ \ \  
 \infer[\rsch_{T}]{\Ga,\mdl\phi \seq\De}{\Ga, \phi \seq \De }
\]
\end{example}

\begin{remark}
 \label{remweight}
If the rules in $\MR$ are terminating in the Dyckhoff order, then \DYX\ is terminating in the Dyckhoff order.
\end{remark}

\subsection{Structural rules}
A calculus $\G$ is {\em closed under weakening} if the following two rules are admissible: 
\[
 \infer[LW]{\Ga,\phi \seq \De}{\Ga \seq \De} \ \ \ \ \infer[RW]{\Ga \seq \phi}{\Ga \seq \ }
\]
It is {\em closed under contraction} if the following rule is admissible: 
\[
 \infer[LC]{\Ga,\phi \seq \De}{\Ga,\phi,\phi \seq \De}
\]
$\G$ has {\em cut-elimination} or is {\em closed under cut} if the Cut rule is admissible:
\[
 \infer[\it Cut]{\Ga_1,\Ga_2 \seq \De}{\Ga_1\seq \phi & \Ga_2,\phi \seq \De}
\]
$\G$ is {\em closed under the structural rules} if it is closed under weakening and contraction and has cut-elimination. 
$\G$ is closed under {\em Implication Inversion} if the following rule is admissible:
\[
 \infer{\Ga, \psi \seq \De}{\Ga,\phi\imp\psi \seq \De}
\]
A rule is {\em invertible} if the derivability of the conclusion implies the derivability of the premises. 

Because of the power of the Cut rule, the following lemma is straightforward.

\begin{lemma}
 \label{leminvertible}
If $\LJX$ is closed under the structural rules, then the rules $R\en$, $L\en$, $L\of$, and $R\!\imp$ are invertible and $\LJX$ is closed under Implication Inversion. 
\end{lemma}

\section{Equivalence of G3iX and G4iX}
 \label{secequivalence}
In this section, let $\LJX$ be a calculus that consists of $\LJ$ plus a set of modal rules $\MR$. Recall that $\DYX$ is the calculus $\DY$ extended by the rules in $\MR$ plus the rules $\rsch^\imp$ for those $\rsch\in \MR$ that are right modal rules. We show that $\DYX$ and $\LJX$ are equivalent, for which we first have to prove a normal form theorem (Lemma~\ref{lemstrict}) for derivations in \LJX.

\subsection{Strict proofs in G3iX}
A multiset is {\em irreducible} if it has no element that is a disjunction or a conjunction or falsum and for no atom $p$ does it contain both $p\imp\psi$ and $p$. A sequent $S$ is {\em irreducible} if $S^a$ is. A proof is {\em sensible} if its  last inference does not have a principal formula on the left of the form $p\imp\psi$ for some atom $p$ and formula $\psi$.\footnote{In \cite{iemhoff18} the requirement that the principal formula be on the left was erroneously omitted.} A proof in $\LJX$ is {\em strict} if in the last inference, in case it is an instance of $L\!\imp$ with principal formula $\mdl\phi \imp \psi$, the left premise is an axiom or the conclusion of an application of a right modal rule. 

Observe that ``strict'' and ``sensible'' are properties of proofs and depend only on the lowest inferences of a proof. Namely, the end of a proof always has one of the following two forms: 
\[
 \infer[\rsch_0]{S_0}{
  \infer[\rsch_1]{S_1}{\vdots}
  & 
  \deduce{S_2}{\vdots}
 }
 \hspace{2cm}
 \infer[\rsch_0]{S_0}{
  \infer[\rsch_1]{S_1}{\vdots}
 }
\]
Here any of the rules $\rsch_i$ may be an axiom, in which case there are no sequents above its conclusion in the picture above. Now whether the proof is strict or sensible depends only on the sequents $S_i$ and the rules $\rsch_i$. It is necessarily strict and sensible if $\rsch_0$ is not $L\!\imp$. In the case that it is,  the principal formula must be of the form $\phi\imp\psi$ and the proof looks as follows:
\[
\] \[
 \infer[L\!\imp]{\Ga,\phi \imp \psi \seq \De}{
  \infer[\rsch_1]{\Ga,\phi \imp \psi \seq\phi}{\vdots}
  & 
  \deduce{\Ga, \psi \seq \De}{}
 }
\]
In case $\phi$ is not atom or a pure modal formula, the proof is strict and sensible. Only when $\phi$ is an atom, the proof is not sensible, but it is strict.  If $\phi$ is a pure modal formula, it is sensible but may not be strict, depending on the rule $\rsch_1$. The proof is strict once $\rsch_1$ is an axiom or a right modal rule, as in the following two examples.

\begin{example}
Two strict proofs in $\LJK$, where the left premise of the last inference is the right modal rule $\rsch_K$ in the leftmost example and an axiom in the other example: 
\[\small
 \infer[L\!\imp]{\mdl p, \mdl p \imp q \seq q}{
  \infer[\rsch_K]{\mdl p, \mdl p \imp q \seq \mdl p}{p \seq p} 
  & 
  \mdl p, q \seq q 
  }
  \ \ \ \ 
   \infer[L\!\imp]{\bot, \mdl p \imp q \seq q}{
  \bot, \mdl p \imp q \seq \mdl p 
  & 
  \bot, q \seq q 
  }
\]
The rightmost example is of course a nonsensical derivation, but strict proofs of this form are such. The important strict proofs are the ones where the left premise of the last inference in the conclusion of a right modal rule, as in the leftmost example. 
\end{example}

\begin{remark}
 \label{remirrseq}
If $\MR$ is nonflat, then in any strict proof ending with an instance of $L\!\imp$ with principal formula $\mdl\phi \imp \psi$ and conclusion $S$, because the formula in the succedent of the left premise is $\mdl\phi$, in case the left premise is an instance of an axiom it can only be an instance of $L\bot$. This implies that if $S$ is irreducible, the left premise cannot be an instance of an axiom and thus is required to be the conclusion of an application of a right modal rule.
\end{remark}

\begin{lemma} 
 \label{lemstrict} 
If $\LJX$ is closed under Implication Inversion, then 
every irreducible sequent that is provable in $\LJX$ has a sensible strict proof in $\LJX$. 
\end{lemma}
\begin{proof}
This is proved in the same way as the corresponding lemma (Lemma 1) in \cite{dyckhoff92}. Arguing by contradiction, assume that among all provable irreducible sequents that have no sensible strict proofs, $S$ is such a sequent with the shortest proof, $\cald$, where the {\em length} of a proof is the length of its leftmost branch. Thus the last inference in the proof is an application 
\[
 \infer{\Ga,\phi\imp \psi \seq \De}{
 \deduce[\cald_1]{\Ga, \phi\imp \psi \seq \phi}{} & 
 \deduce[\cald_2]{\Ga, \psi \seq \De}{} }
\]
of L$\imp$, where $\phi$ is an atom or a pure modal formula.  
Since $S^a$ is irreducible, $\bot \not\in S^a$ and if $\phi$ is an atom, $\phi\not\in S^a$. Therefore the left premise cannot be an instance of an axiom and hence is the conclusion of a rule, say $\rsch$. Since the succedent of the conclusion of $\rsch$ consists of an atom or a pure modal formula, $\rsch$ is a left rule or a right modal rule. The latter case cannot occur, since the proof then would be strict and sensible. Thus $\rsch$ is a left rule. 

We proceed as in \cite{dyckhoff92}. Sequent $(\Ga, \phi\imp \psi \seq \phi)$ is irreducible and has a shorter proof than $S$. Thus its subproof $\cald_1$ is strict and sensible. Since the sequent is irreducible and $\phi$ is an atom or a pure modal formula, the last inference of $\cald_1$ is L$\imp$ with a principal formula $\phi'\imp \psi'$ such that $\phi'$ is not an atom. Let $\cald'$ be the proof of the left premise $(\Ga, \phi\imp \psi \seq \phi')$. Thus the last part of $\cald$ looks as follows, where $\Pi,\phi'\imp \psi'=\Ga$.  
\[
 \infer{\Pi, \phi\imp \psi,\phi'\imp \psi' \seq \De}{
 \infer{\Pi, \phi\imp \psi,\phi'\imp \psi' \seq \phi}{
  \deduce[\cald']{\Pi, \phi\imp \psi,\phi'\imp \psi' \seq \phi'}{} & 
  \deduce[\cald'']{\Pi, \phi\imp \psi,\psi' \seq \phi}{}} & 
 \deduce[\cald_2]{\Pi, \psi,\phi'\imp \psi' \seq \De}{} }
\]
Consider the following proof of $S$.  
\[
 \infer{\Pi, \phi\imp \psi,\phi'\imp \psi' \seq \De}{
  \deduce[\cald']{\Pi, \phi\imp \psi,\phi'\imp \psi' \seq \phi'}{} & 
  \infer{\Pi, \phi\imp \psi,\psi'\seq \De}{ 
   \deduce[\cald'']{\Pi, \phi\imp \psi,\psi' \seq \phi}{} & 
   \deduce[\cald''']{\Pi,\psi,\psi' \seq \De}{} 
  } 
 }
\]
The existence of $\cald'''$ follows from the assumption of Implication Inversion and the existence of $\cald_2$. The obtained proof is strict and sensible: In case $\phi'$ is not a pure modal formula, this is straightforward. In case $\phi'$ is a pure modal formula, it follows from the fact that was observed above, namely that $\cald_1$ is strict and sensible. 
\end{proof}

\subsection{Equivalence Theorem}

\begin{theorem} 
 \label{thmequivalence}
If $\LJX$ is nonflat and closed under the structural rules, and $\DYX$ is terminating and closed under weakening and contraction, then \LJX\ and \DYX\ are equivalent (derive exactly the same sequents). 
\end{theorem}

\textcolor{cyan}{In an earlier arXiv version of this paper the word ``contraction'' was omitted, without that, the proof of the theorem fails. The published version of this paper in Studia Logica contains the theorem as formulated above, so the correct version.}

\begin{proof}
The proof is an adaptation of the proof of Theorem 1 in \cite{dyckhoff92} and a generalization of the proofs of similar statements in \cite{iemhoff18}. 
Under the assumptions in the theorem we have to show that for all sequents $S$: 
$\af_{\LJX} S$ if and only if $\af_{\DYX} S$. 

The proof of the direction from right to left is straightforward because $\LJX$ is closed under the structural rules, but let us fill in some of the details. We use induction to the height of the proof of a sequent in $\DYX$, where the {\em height} of a derivation is the length of its longest branch, where the length of a branch is the number of nodes it contains. Thus under this convention proofs that consist of a single axiom have height one. 

Suppose $\af_{\DYX}S$. If $S$ is an instance of an axiom, then clearly $\af_{\LJX}S$ as well. Suppose $S$ is not an instance of an axiom and consider the last inference of the proof of $S$. We distinguish according to the rule $\rsch$ of which the last inference is an instance. 

If $\rsch$ is a rule that belongs to $\LJX$ we reason as follows. The premises of the instance are both derivable in $\LJX$ by the induction hypothesis. Since $\rsch$ belongs to $\LJX$, this implies that $S$ is derivable in $\LJX$ as well. 

It remains to treat the case that $\rsch$ does not belong to $\LJX$, which means that it is one of the rules $L\en\!\imp$, $L\of\!\imp$, $L\!\imp\!\imp$ or $Lp\!\imp$. 

In the case of $L\en\!\imp$, $S$ is of the form $\Ga, \phi\en\psi \imp \gamma \seq \De$ and the premise of the rule is $\Ga,\phi\imp (\psi\imp\gamma)\seq\De$. The premise is derivable in $\LJX$ by the induction hypothesis. Since $\phi\en\psi \imp \gamma \seq \phi\imp (\psi\imp\gamma)$ is clearly derivable in \LJ, the closure under cut implies that $S$ is derivable in $\LJX$. The rule $L\of\!\imp$ can be treated in the same manner. 

If the rule is $Lp\!\imp$, then $S$ is of the form $\Ga,p,p\imp \phi \seq \De$. The premise is 
$\Ga,p,\phi \seq \De$, which, by the induction hypothesis, is derivable in $\LJX$. It is not hard to show that $\Ga,p,p\imp \phi\seq \phi$ is also derivable in $\LJX$. Closure under Cut and Contraction shows that so is $S$.   

If the rule is $L\!\imp\!\imp$, then $S$ is of the form $\Ga, (\phi \imp \psi)\imp \gam \seq \De$
and the premisses are $S_1=(\Ga,\psi\imp\gam\seq \phi\imp\psi)$ and $S_2=(\Ga,\gam\seq \De)$, which are derivable in $\LJX$ by the induction hypothesis. It is not difficult to show that $S_3=(\Ga,(\phi \imp \psi) \imp \gam, \phi \seq \psi \imp \gam)$ is derivable in $\LJ$, and thus in $\LJX$. The admissibility of contraction and cut, applied to $S_1$ and $S_3$, implies that $\Ga, (\phi \imp \psi)\imp \gam, \phi \seq \phi \imp \psi$ is derivable in $\LJX$. Hence so is $\Ga, (\phi \imp \psi)\imp \gam \seq \phi \imp (\phi \imp \psi)$. Since $\phi \imp (\phi \imp \psi) \seq \phi \imp \psi$ is derivable in $\LJ$, closure under Cut gives $S_4=(\Ga, (\phi \imp \psi)\imp \gam \seq\phi \imp \psi)$.  
An application of $L\!\imp$ to $S_4$ and $S_2$ proves that $S$ is derivable in $\LJX$. 
 
The other direction, left to right, is proved by induction on the order $\sml$ with respect to which \DYX\ is terminating. So suppose $\LJX\af S$. 
Sequents lowest in the order do not contain connectives or modal operators by definition of the weight function underlying $\sml$. Since the calculi are nonflat, such sequents have to be instances of axioms, and since $\LJX$ and $\DYX$ have the same axioms, $S$ is provable in $\DYX$. 

We turn to the case that $S$ is not the lowest in the order. 
If $S^a$ contains a conjunction,  $S = (\Ga,\phi_1\en\phi_2 \seq \De)$, then $S'=(\Ga,\phi_1,\phi_2\seq \De)$ is provable in $\LJX$ by Lemma~\ref{leminvertible}. As \DY\ contains $L\en$ and \DYX\ is terminating, $S'\sml S$ follows. Hence $S'$ is provable in $\DYX$ by the induction hypothesis. Thus so is $(\Ga,\phi_1\en\phi_2 \seq \De)$. A disjunction in $S^a$ as well as the case that both $p$ and $p\imp \phi$ belong to $S^a$, can be treated in the same way. 

Thus only the case that $S$ is irreducible remains, and by Lemmas~\ref{leminvertible} and \ref{lemstrict} we may assume its proof in \LJX\ to be sensible and strict. The irreducibility of $S$ implies that the last inference of the proof is an application of a rule, $\rsch$, that is either a nonmodal right rule, a modal rule or $L\!\imp$. In the first two cases, $\rsch$ belongs to both calculi and the fact that $\DYX$ is terminating implies that the premise(s) of $\rsch$ is lower in the order $\sml$ than $S$. Thus the induction hypothesis implies that the premise(s) is derivable in $\DYX$, and since $\rsch$ belongs to $\DYX$, the conclusion $S$ is derivable in $\DYX$ as well. 

We turn to the third case. 
Suppose that the principal formula of the last inference is $(\gam \imp \psi)$ and $S = (\Ga,\gam\imp\psi \seq \De)$. Since the proof is sensible, $\gam$ is not atomic. 
We distinguish according to the main connective of $\gam$.  

If $\gam = \bot$, then $\Ga\seq \De$ is derivable in $\LJX$ because of the closure under cut: $\LJ$ derives $(\ \seq \bot \imp \psi)$, and so the cut 
\[
 \infer{\Ga\seq \De}{\ \seq \bot \imp \psi & \Ga,\bot\imp\psi \seq \De}
\]
shows that $\Ga\seq \De$ is derivable in $\LJX$. Since $(\Ga\seq \De) \sml S$, it follows that $\Ga\seq \De$ is derivable in $\DYX$ by the induction hypothesis. As $\DYX$ is closed under weakening, $S$ is derivable in $\DYX$ too. 

If $\gam = \phi_1 \en \phi_2$, then because $\gam\imp\psi$ is the principal formula, the premise of the last inference is $S'=(\Ga,\phi_1 \imp (\phi_2\imp \psi) \seq \De)$. Thus $S'$ is derivable in $\LJX$, and  
the fact that $\DYX$ is terminating and contains $L\en\!\!\imp$ implies $S'\sml S$. Hence $S'$ is derivable in $\DYX$ by the induction hypothesis. Thus so is 
$\Ga,\phi_1 \en \phi_2 \imp \psi\seq \De$ by an application of $L\en\!\!\imp$. The case that $\gam = \phi_1 \of \phi_2$ is analogous.

If $\gam = \phi_1 \imp \phi_2$, then because $\gam\imp\psi$ is the principal formula, 
both premises $S_1=(\Ga, \psi \seq \De)$ and $\Ga, \gam \imp \psi \seq \gam$ are derivable in $\LJX$. Thus so is $\Ga, \gam \imp \psi,\phi_1 \seq \phi_2$ by the invertibility of $R\!\imp$ (Lemma~\ref{leminvertible}). 
Consider the following inference steps, where $\varchi$ is short for $\phi_2\imp\psi$:
\[\small 
 \infer[R\!\imp]{\Ga, \varchi\seq \gamma}{
 \infer[LC]{\Ga,\varchi,\phi_1 \seq \phi_2}{
  \infer[\rm Cut]{\Ga,\varchi,\phi_1,\phi_1 \seq \phi_2}{
   \infer[R\!\imp]{\varchi,\phi_1 \seq \gamma\imp \psi}{
    \infer[L\!\imp]{\varchi,\phi_1,\gamma \seq \psi}{
     \infer[L\!\imp]{\varchi,\phi_1,\gamma \seq \phi_2}{
      \infer*{\varchi,\phi_1,\gamma \seq \phi_1}{} 
      &
      \infer*{\varchi,\phi_1,\phi_2 \seq \phi_2}{} }
     & 
     \infer*{\psi,\phi_1,\gamma \seq \psi}{} } } 
   & 
   \infer*{\Ga,\gamma\imp\psi,\phi_1 \seq \phi_2}{} } } 
 }
\]
As just observed, the right premise of the cut is derivable in $\LJX$. It is clear that the left premise of the cut can be extended into a derivation in $\LJ$ (we have provided the last steps and leave the prove that $\Pi,\phi \seq \phi$ is derivable in $\LJ$, for any $\Pi$ and $\phi$, to the reader). Thus both premises of the cut are derivable in $\LJX$, and therefore the conclusion $\Ga,\varchi,\phi_1,\phi_1 \seq \phi_2$ is too. As $\LJX$ is closed under contraction and contains $R\!\imp$, the last sequent 
$S_2=(\Ga,\phi_2\imp\psi \seq \gamma)$ is also derivable in $\LJX$. 
Since \DYX\ is terminating and $S_1$ and $S_2$ are the premises of $L\!\imp\!\imp$, they both are lower in the order $\sml$ than $S$. Therefore they are 
derivable in $\DYX$ by the induction hypothesis. And thus so is $S$ by an application of $L\!\imp\!\imp$. 

If $\gam = \mdl\phi$, then Remark~\ref{remirrseq} and the fact that the proof is strict and $S$ is irreducible implies that the left premise is the conclusion of an application of a right modal rule $\rsch$ with premises $S_1, \dots, S_n$. Thus the derivation looks as follows:
\[
 \infer{\Ga,\mdl\phi\imp \psi \seq \De}{
  \infer[\rsch]{\Ga,\mdl\phi\imp \psi \seq \mdl\phi}{
   \deduce[\cald_1]{S_1}{} & \dots & \deduce[\cald_n]{S_n}{}}  & 
  \deduce[\cald_0]{\Ga,\psi \seq \De}{}}
\]
Therefore \DYX\ contains the rule 
\[
  \infer[\rsch^\imp]{\Ga,\mdl\phi\imp \psi,\mdl\phi\imp \psi \seq \De}{
  \deduce[]{S_1}{ } & \dots & \deduce[]{S_n}{ } & 
  \deduce[]{\Ga,\psi\seq \De}{} }
\]
Since \DYX\ is terminating, it follows that $(\Ga,\psi \seq \De) \sml S$ and $S_i \sml S$. 
By the induction hypothesis, the $S_i$ as well as $\Ga,\psi \seq \De$ are derivable in $\DYX$, say with derivations $\cald_i'$ and $\cald_0'$, respectively. Since $\rsch^\imp$ belongs to the calculus, the following is a proof of $\Ga,\mdl\phi\imp \psi,\mdl\phi\imp \psi \seq \De$ in $\DYX$:
\[
  \infer[\rsch^\imp]{\Ga,\mdl\phi\imp \psi,\mdl\phi\imp \psi \seq \De}{
  \deduce[\cald_1']{S_1}{ } & \dots & \deduce[\cald_n']{S_n}{ } & 
  \deduce[\cald_0']{\Ga,\psi\seq \De}{} }
\]
Closure under contraction implies that $\Ga,\mdl\phi\imp \psi \seq \De$ is derivable in $\DYX$. 
\end{proof}

\begin{corollary} 
 \label{corconcutdy} 
If $\LJX$ is nonflat and closed under the structural rules, and $\DYX$ is terminating and closed under weakening and contraction, then the structural rules are admissible in $\DYX$.   
\end{corollary}

Theorem~\ref{thmequivalence} and Corollary~\ref{corconcutdy} combined with  Remark~\ref{remweight} give the following.
 
\begin{corollary} 
 \label{corduckhofforder} 
If $\LJX$ is nonflat and closed under the structural rules, all rules not in $\LJ$ are terminating in the Dyckhoff order, and $\DYX$ is closed under weakening and contraction, then \LJX\ and \DYX\ are equivalent, \DYX\ is terminating, and the structural rules are admissible in $\DYX$.  
\end{corollary}

\section{Conclusion}
 \label{secconclusion}
It has been shown that for any calculus \LJX\ that consists of \LJ\ plus a set of nonflat (multi-)modal rules $\MR$, there exists a calculus \DYX\ that is equivalent to \LJX, provided \LJX\ is nonflat and closed under the structural rules and \DYX\ is terminating and closed under weakening and contraction. In the setting of intuitionistic modal logics, one usually requires one's sequent calculi to be closed under the structural rules, so that requirement on \LJX\ is relatively innocent. Likewise for the closure under weakening and contraction of \DYX. The requirement that \DYX\ be terminating is of a different kind. Although many common modal rules are terminating in the Dyckhoff order, some well-known rules are not. Examples are the standard rules for transitivity, G\"odel-L\"ob Logic and Strong L\"ob Logic (writing $\bx$ for $\mdl$): 
\[
 \infer[\rsch_{K4}]{\Pi,\bx\Ga \seq \bx\phi}{\Ga,\bx\Ga \seq \phi} \ \ \ \ 
 \infer[\rsch_{GL}]{\Pi,\bx\Ga \seq \bx\phi}{\Ga,\bx\Ga,\bx\phi \seq \phi} \ \ \ \ 
 \infer[\rsch_{SL}]{\bx\Sig,\Pi,\bx\Ga \seq\bx\phi}{\Pi,\bx\Ga,\Ga,\bx \phi \seq \phi}
\]
Therefore, to apply the method in this paper to such logics, an order different from the Dyckhoff order has to be found with respect to which \DYX\ is terminating. G\"odel-L\"ob Logic and Strong L\"ob Logic are examples for which that can be done, as shown in \cite{giessen&iemhoff2019,giessen&iemhoff2020} using an ingenious order introduced by B\'ilkov\'a \cite{bilkova06}. For reasons that are out of the scope of this paper we conjecture that for $\DYKf$ and $\DYSf$ such an order does not exist but leave it for future work to settle that conjecture.

\paragraph{Acknowledgements.} 
I thank Iris van der Giessen for the pleasant collaboration that lead to two papers \cite{giessen&iemhoff2019,giessen&iemhoff2020} that were the inspiration for the material presented in this paper. 
I thank an anonymous referee for the useful remarks on an earlier draft of this paper.
Support by the Netherlands Organisation for Scientific Research under grant 639.073.807 is gratefully acknowledged.


\end{document}